\documentclass{amsart}
\usepackage[dvipdfmx]{graphicx}
\usepackage{amscd}
\usepackage{color}
\usepackage{amsmath,amssymb,amsthm}

\newtheorem{Def}{Definition}[section]
\newtheorem{thm}{Theorem}[section]

\newtheorem{lem}{Lemma}[section]
\newtheorem{prop}{Proposition}[section]
\newtheorem{claim}{Claim}[section]

\begin{document}

\title[ Trisections with Kirby-Thompson length 2]{Trisections with Kirby-Thompson length 2}
\author{Masaki Ogawa}
\date{}
\begin{abstract}
Kirby and Thompson introduced a length of a trisection in \cite{KT}.
They also defined the length of a 4-manifold as the minimum of length among all lengths of trisection of a 4-manifold.
In this paper, we consider trisections whose Kirby-Thompson length is 2.
Kirby and Thompson conjectured that length 2 trisection is a trisection of 4-manifold with length 0. We shall prove this conjecture in this paper.
\end{abstract}

\maketitle
\section{introduction}
 Gay and Kirby introduced a decomposition of a 4-manifold with three four-dimensional handlebodies called a trisection \cite{GK}. This is studied by many authors and applied to three-dimensional topology and knot theory.
 
 In this paper,  we consider the invariant $L$ of a closed 4-manifold which is introduced by Kirby and Thompson in \cite{KT}.
 They define the invariant $L_\mathcal{T}$ of a trisection $\mathcal{T}$ by using the curve complex.
 The invariant of a closed, oriented, 4-manifold $L$ is defined by taking the minimum of the invariant of a trisection $L_\mathcal{T}$ and classifies the 4-manifolds which satisfies that  $L=0$. 
 Also, they consider a trisection that satisfies that $L_\mathcal{T}=1$ and showed it is a trisection of a 4-manifolds with $L_X=0$.
 Following this study, a relative version is studied in \cite{rel}. In the paper, the authors defined an invariant of a compact 4-manifold $r\mathcal{L}$ and show $r\mathcal{L}(X)=0$ if and only if $X\cong B^4$.
 
 Bridge trisection, introduced by Meier and Zupan, is a pair of trisection and a knotted surface that is decomposed into some disks \cite{MZ1, MZ2}.
 This is analogous to a bridge splitting of a knot in a 3-manifold.
For the knotted surfaces in 4-manifolds, following the invariant $L$, the invariant of knotted surface $\mathcal{L}$ is defined and studied in \cite{Br}.
From the above results, we can see that the above invariants are one measure of their complexity.
On the other hand, these invariants are difficult to calculate.

 For such an invariant, Kirby and Thompson conjectured that a trisection of closed-oriented 4-manifolds that satisfies that $\mathcal{L}_\mathcal{T}=2$ is also a trisection of a 4-manifold with $\mathcal{L}=0$.
 In this paper, we shall prove this conjecture by using a Kirby diagram induced from a trisection diagram. The result is the following.
 
 \begin{thm}\label{thm1}
 	If there exists a trisection $\mathcal{T}$ such that $L_{X, \mathcal{T}}=2$, then $L_X=0$ and $X$ is diffeomorphic to a connect sum of copies of $S^1\times S^3$, $S^2\times S^2$, $\pm\mathbb{C}P^2$, and $S^4$.
 \end{thm}
 
 To show this theorem, we consider the Kirby diagram induced by a trisection diagram.
 So we construct a Kirby diagram from a trisection diagram.
 The framing coefficient of the component of the Kirby diagram obtained from the trisection diagram can be calculated by using an algebraic intersection number.

 This paper is organized as follows.
 In Section 2, we review the definition of a trisection and its Kirby-Thompson length by using a curve complex.
 In Section 3, following \cite{GK}, we review the construction of a Kirby diagram from a trisection diagram.
After that, we shall prove the main result by dividing our situation into some cases in Section 4.



 \section{A trisection and its length}
 
 First, we review the definition of a trisection and its length $\mathcal{L}_{\mathcal{T}}$.
 Roughly speaking, a trisection is a decomposition of a 4-manifold into three 4-dimensional handlebodies.
 This is analogous to a Heegaard splitting of a 3-manifold.
 \begin{Def}
 	A $(g; k_1, k_2, k_3)$-trisection of a closed, oriented 4-manifold $X$ is a decomposition $X=X_1\cup X_2\cup X_3$ where $X_i\cong \natural^{k_i}S^1\times B^3$, $X_i\cap X_j\cong \natural^g S^1\times B^2$ for $i\neq j$ and $X_1\cap X_2\cap X_3 \cong \#^g S^1\times S^1$.
 \end{Def}
A trisection is said to be balanced when $k_1=k_2=k_3$ is satisfied.
For a trisection, we can define its diagram called a trisection diagram.
A trisection diagram is a 4-tuple of a set of simple closed curves and a closed surface.
\begin{Def}
	A $(g; k_1, k_2, k_3)$ trisection diagram is a 4-tuple $(\Sigma; \alpha, \beta, \gamma)$ such that each of $(\Sigma; \alpha, \beta)$, $(\Sigma; \beta, \gamma)$ and $(\Sigma; \gamma, \alpha)$ is a Heegaard  diagram of $\#^{k_i}S^1\times S^2$ , $i=1, 2, 3$ respectively
\end{Def}
Next, we review the definition of the length of trisection. A {\it cut system} of a genus $g$ closed surface is a union of $g$ essential simple closed curves which cut open the closed surface into a $2g$-punctured sphere.
\begin{Def}
The cut complex $\mathcal{C}$ is a 1-complex whose vertices correspond to an isotopy class of cut system.
 There are two types of edges: one is a type 0 and the other is a type 1.
 Two vertices $\alpha=\{\alpha_1, . . . , \alpha_g\}$, $\alpha'=\{\alpha_1', . . . , \alpha_g'\}$ are connected by a type 0 edge if $\alpha$ and $\alpha'$ agree on $g-1$ curves and their final curves
are disjoint.
 Two vertices $\alpha=\{\alpha_1, . . . , \alpha_g\}$, $\alpha'=\{\alpha_1', . . . , \alpha_g'\}$ are connected by a type 1 edge if $\alpha$ and $\alpha'$ agree on $g-1$ curves and their final curves
intersect at exactly one point.
\end{Def}

Following Kirby and Thompson \cite{KT}, we also use $\Gamma_\alpha$, $\Gamma_\beta$ and $\Gamma_\gamma$ for a set of all vertices in a cut complex that are path connected to $\alpha$, $\beta$ and $\gamma$ by only type 0 edges respectively.
Also, if $\alpha_i$ and $\beta_i$ are parallel, we write $\alpha_i\mathcal{P}\beta_j$, if $\alpha_i$ and $\beta_i$ intersects exactly one point, we write $\alpha_i\mathcal{D}\beta_j$.
\begin{Def}
	We say two cut systems $\alpha$ and $\beta$ are in good position with respect to each other if we can order each, $\alpha=\alpha_1, . . . , \alpha_g$, $\beta=\beta_1, . . . , \beta_g$, so that for each $i$, either $\alpha_i$ is parallel to $\beta_i$ or $\alpha_i$ intersects $\beta_i$ at exactly one point and $\alpha_i\cap\beta_j=\emptyset$ for $i\neq j$.
	$\alpha_i$ and $\beta_j$ are a good pair if they are either parallel or intersect at exactly one point and disjoint from all other curves in $\alpha$  and $\beta$.
\end{Def}
\begin{Def}
	Let $l_{X, \mathcal{T}}$ be the length of the shortest closed path in $\mathcal{C}$ that intersects each of $\alpha$, $\beta$, and $\gamma$, which also satisfies the following:
	\begin{enumerate}
		\item[(1)] there are pairs $(\alpha_\beta, \beta_\alpha), (\beta_\gamma, \gamma_\beta)$ and $(\gamma_\alpha, \alpha_\gamma)$ in $(\Gamma_\alpha, \Gamma_\beta)$, $(\Gamma_\beta, \Gamma_\gamma)$ and $(\Gamma_\gamma, \Gamma_\alpha)$ respectively, which are all good, so it take $g-k_i$ type 1 moves to travel from the vertex corresponding to one element in the pair to the other.
		\item[(2)] the subpath of $l_{X, \mathcal{T}}$ connecting $\alpha_\beta$ to $\alpha_\gamma$ (resp. $\beta_\alpha$ to $\beta_\gamma$, $\gamma_\alpha$ to $\gamma_\beta$) remains within $\Gamma_\alpha$ (resp. $\Gamma_\beta$, $\Gamma_\gamma$).
	\end{enumerate}
\end{Def}

\begin{Def}
	$L_{X, \mathcal{T}}=l_{X, \mathcal{T}}-3g+k_1+k_2+k_3$.
\end{Def}
\begin{Def}
	The length $L_X$ of  a closed 4-manifold $X$ is the minimum value of $L_{X, \mathcal{T}}$ over all trisections $\mathcal{T}$ of $X$.
\end{Def}
For more details about the length of a 4-manifold, see \cite{KT}.
Kirby and Thompson showed that if $L_X=0$, $X$ is diffeomorphic to a connected sum of copies of $S^1\times S^3$, $S^2\times S^2$, $\pm\mathbb{C}P^2$  and $S^4$ in \cite{KT}.
We shall name the shortest path which realizes the length of a trisection $\mathcal{T}$.
 \begin{Def}
 	 A realization loop of a trisection $\mathcal{T}$ is a loop in a cut complex that realizes the length of $\mathcal{T}$.
 \end{Def}
 Let $\mathcal{T}$ be a trisection, $(\Sigma, \alpha, \beta, \gamma)$  a trisection diagram of $\mathcal{T}$ and $l$ a realization loop of $\mathcal{T}$.
 If the length of trisection $\mathcal{T}$  equals 2, it will be sufficient to consider the following cases.

 \begin{enumerate}
 	\item $\Gamma_\gamma\cap l$ consists of two edge and $\alpha=\alpha_\beta=\alpha_\gamma$ and $\beta=\beta_\alpha=\beta_\gamma$.
	\item both $\Gamma_\gamma\cap l$ and $\Gamma_\beta\cap l$ are consists of one edge and $\alpha=\alpha_\beta=\alpha_\gamma$.
 \end{enumerate}
 
	 \begin{figure}[h]
		\centering
			\includegraphics[scale=0.6]{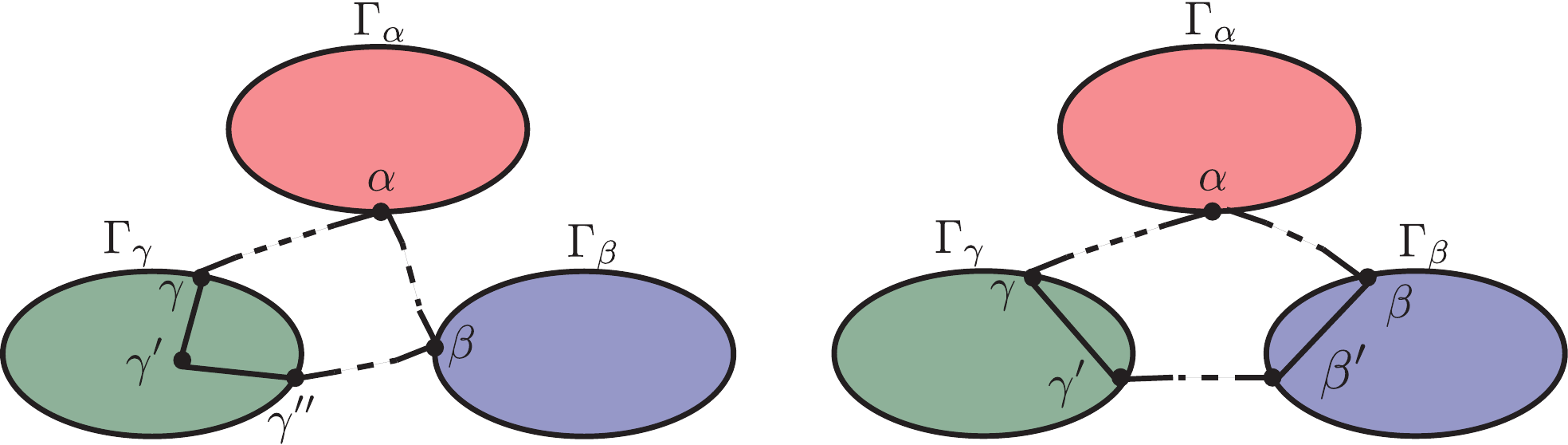}
			\caption
			{The realizing loop will be  either right or left  figure if the length of a trisection is two.}
			\label{typeia}
	\end{figure}
	
The case (1) above can be further divided into two cases; We consider as $\gamma=(\gamma_1, . . . , \gamma_g)$ and $\gamma'=(\gamma_1, . . . , \gamma_{g-1}, \gamma_g')$.
Then there are two cases $\gamma''=(\gamma_1, . . . , \gamma_{g-1}', \gamma_g')$ or $\gamma''=(\gamma_1, . . . , \gamma_{g-1}, \gamma_g'')$ where $\gamma'$ is length one from $\gamma$, $\gamma''$ is length one from $\gamma'$ and length two from $\gamma$. 

 \section{A Kirby diagram from trisection diagram}
 In this section, we review the construction of a Kirby diagram from a trisection diagram. 
 We can obtain the 4-manifold which is represented by a trisection diagram $(\Sigma; \alpha, \beta, \gamma)$ by attaching several 2-handles along $\alpha$, $\beta$ and $\gamma$ to $\Sigma\times D^2$.
 
 We consider a trisection diagram $(\Sigma; \alpha, \beta, \gamma)$ which satisfies that  $(\Sigma, \alpha, \beta)$ is a standard Heegaard diagram of $S^3$.
This implies that $X_1$ is a 4-ball and $\Sigma$ is standardly embedded in $\partial X_1$.
 After cutting open $\Sigma$ along $\alpha$ and $\beta$, we obtain a g punctured sphere $P$. 
We name the boundary of $P$ as $\partial _i$ when it is induced by $\alpha_i$ and $\beta_i$.
We sometimes represent $P$ as a g punctured $\mathbb{R}^2$.
We consider the orientation of $\Sigma$ as an orientation induced by the orientation of $P$.
Suppose $\mathbb{R}^2$ has a standard orientation and the orientation of $P$ is an orientation induced by the standard orientation of $\mathbb{R}^2$ in this paper.
Let $C_1$ and $C_2$ be a simple closed curves on $\Sigma$.
Then the algebraic intersection number of $C_1$ and $C_2$ in $\Sigma$ is denoted by $i(C_1, C_2)$.
See  Section 6 of \cite{FM} for more details about the algebraic intersection number.
Then we shall orient $\alpha_i$ and $\beta_i$ so that $i(\alpha_i, \beta_i)=1$  (See Figure \ref{P}).
	\begin{figure}[ht]
		\centering
			\includegraphics[scale=0.6]{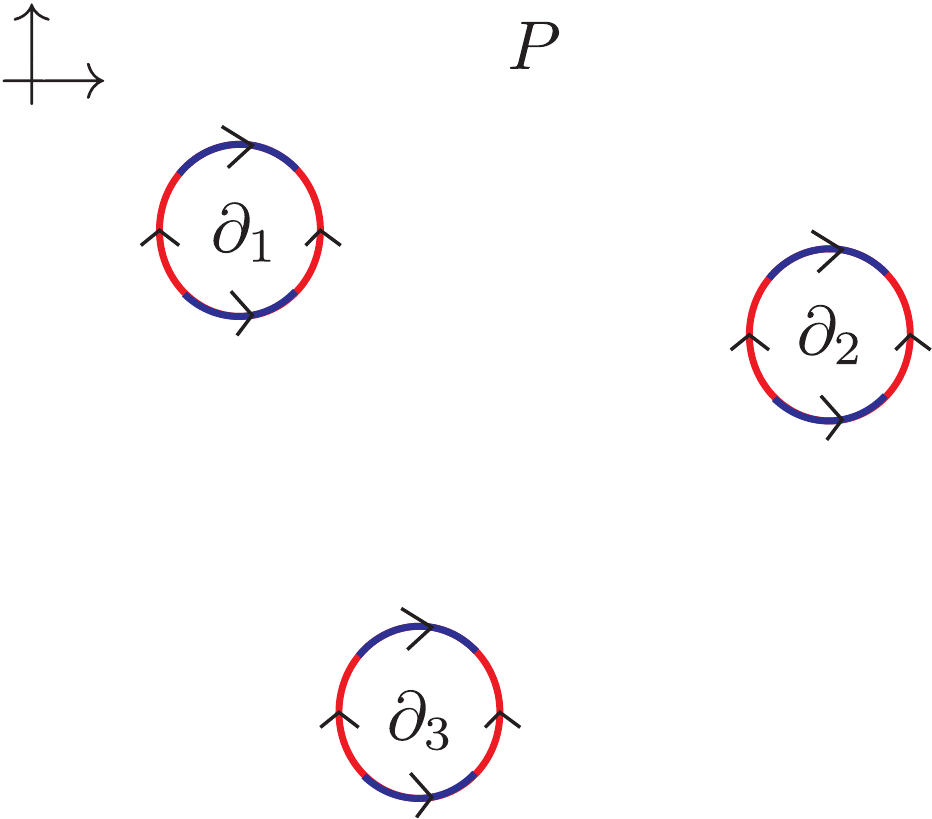}
			\caption
			{We fix the orientation of $P$ as in this figure in this paper and orient $\alpha_i$ and $\beta_i$ as in this picture.}
			\label{P}
	\end{figure}
We call the basis $\{ \alpha_1, \beta_1, . . . . , \alpha_g, \beta_g \}$ of $H_1(\Sigma)$ whose orientation respects the orientation of $P$ a {\it standard basis} in this paper.

We shall now describe the construction of the Kirby diagram.
we consider the surface $\Sigma^{\circ}=(\Sigma, \alpha, \beta)-D$ where $D$ is a disk in $\Sigma$ which is disjoint from $\alpha$, $\beta$ and $\gamma$ (See Figure \ref{TD_1}).
 
  	 \begin{figure}[h]
		\centering
			\includegraphics[scale=0.6]{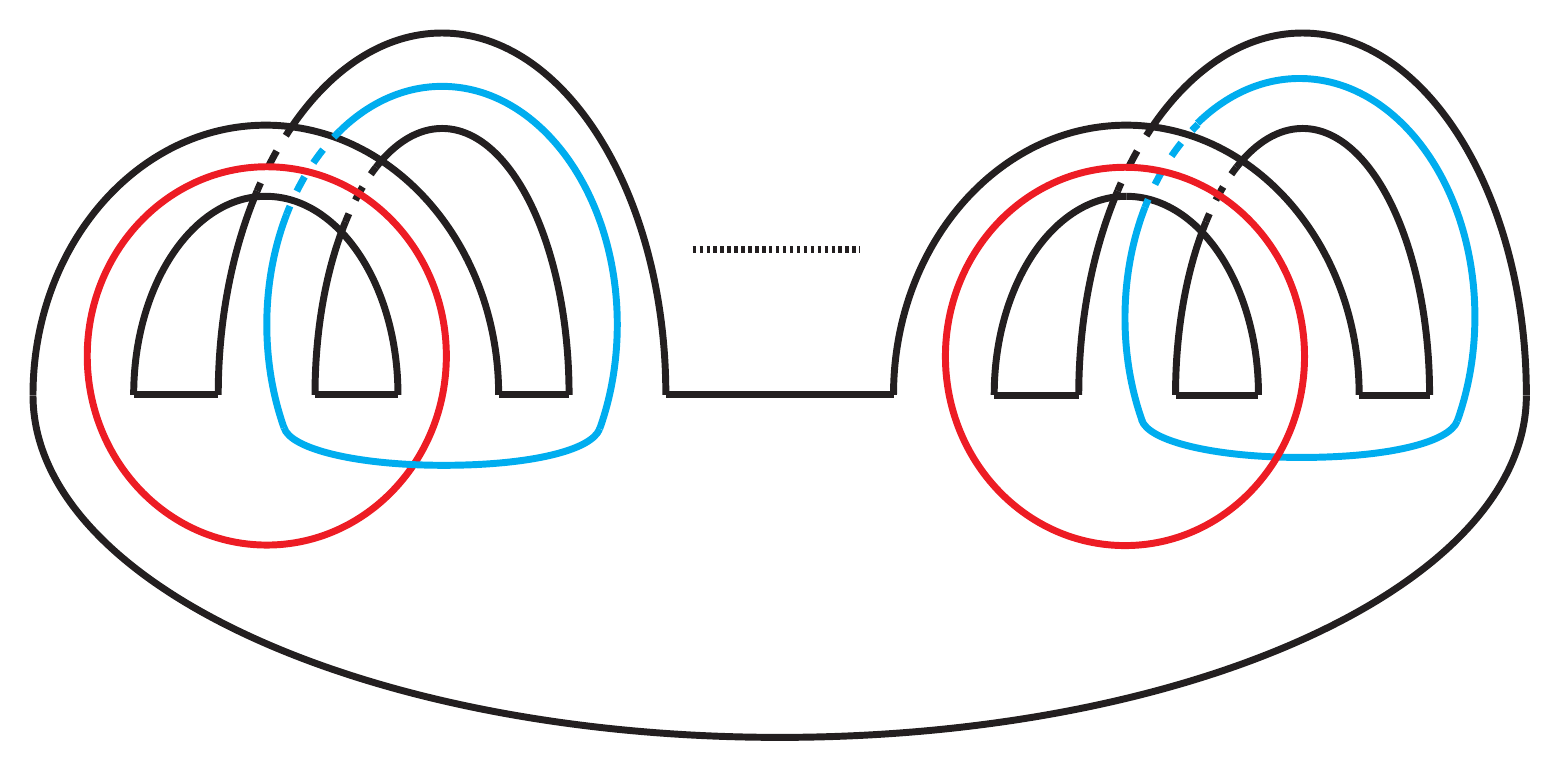}
			\caption
			{The punctured standard genus g Heegaard diagram of $S^3$.}
			\label{TD_1}
	\end{figure}
A 4-manifold $X$ represented by $(\Sigma; \alpha, \beta, \gamma)$ is obtained by attaching 4-dimensional 2-handles to $\Sigma\times D^2$ along $\alpha$, $\beta$ and $\gamma$.
This operation corresponds to that erase $\alpha$ and $\beta$ from the diagram and we consider $\gamma$ as a Kirby diagram. See \cite{GK} for more detail. 
We can assume that $\gamma$ is in $\Sigma^\circ$.
If $\gamma_i$ and $\gamma_j$ intersect $\alpha_i$ and $\beta_i$ exactly once respectively, then they make a crossing in a diagram of $L$.
In such crossing   $\gamma_i$ is the lower arc and $\gamma_j$ is the upper arc (See Figure \ref{crossing}).
All crossing in the Kirby diagram of $X$ obtained by $(\Sigma; \alpha, \beta, \gamma)$ is induced by this process. 
We note that the framings of the components of a Kirby diagram are induced by the surface framing.

  	 \begin{figure}[h]
		\centering
			\includegraphics[scale=0.6]{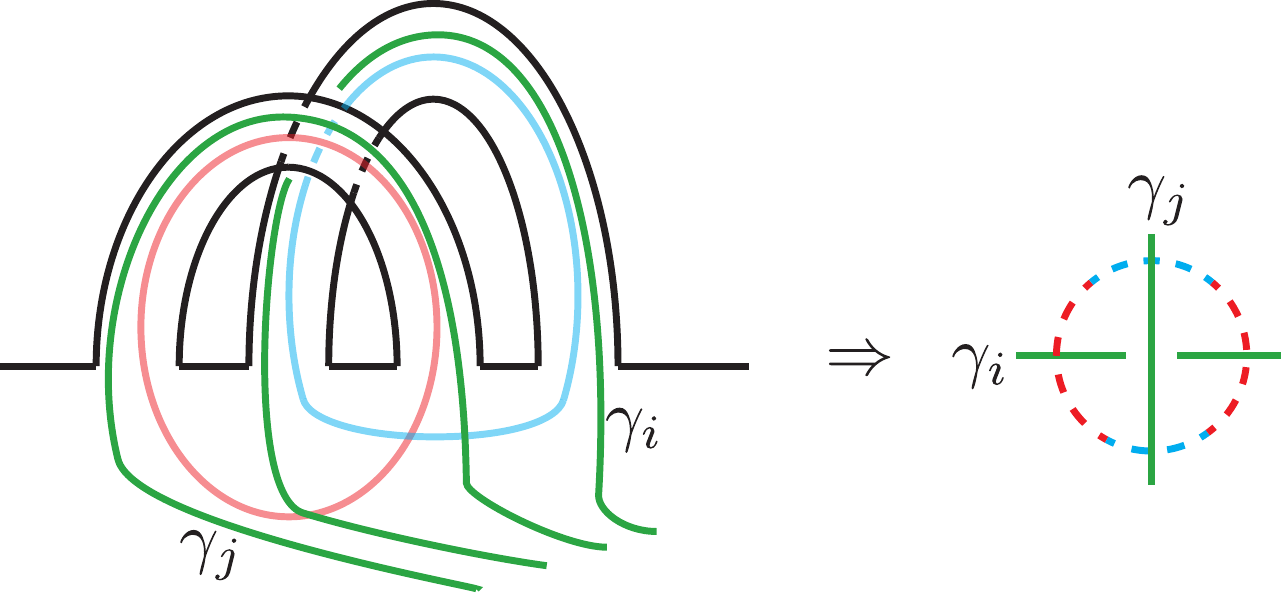}
			\caption
			{The crossings in the Kirby diagram of $X$ were obtained from a trisection diagram.}
			\label{crossing}
	\end{figure}
\begin{lem}\label{unknot}
	Let $L$ be a Kirby diagram obtained from a trisection diagram $(\Sigma; \alpha, \beta, \gamma)$.
	If either $\gamma_i\cap \alpha=\gamma_i\cap \alpha_j$ or $\gamma_i\cap \beta=\gamma_i\cap \beta_j$ is exactly one point, then $\gamma_i$ is an unknot component of $L$.
\end{lem}
\begin{proof}
	Let $D$ be a disk in $\Sigma$ which is disjoint from $\alpha$, $\beta$ and $\gamma$ and $(h^{\alpha}_1, . . . , h^{\alpha}_g)$ and $(h^{\beta}_1, . . . , h^{\beta}_g)$ 2-dimensional 1-handles of $\Sigma-D$ which contain subarcs of $\alpha$ and $\beta$ respectively.
	Then we attach 2-dimensional 2-handles $(D^{\alpha}_1, . . . , D^{\alpha}_g)$ and $(D^{\beta}_1, . . . , D^{\beta}_g)$ to $\partial(\Sigma-D)$ so that $\partial D^{\alpha}_j$ and $\partial D^{\beta}_j$ isotopic to $\alpha_j$ and $\beta_j$ in $\Sigma-D$ respectively for $j\in\{1, . . . , g\}$.
	We consider the case where $\gamma_i\cap \beta=\gamma_i\cap \beta_j$.
	Then $\gamma_i$ is embedded in 
	\[
		(\Sigma-D)-\bigcup_{i\neq j} h^{\alpha}_i
	\]
	since $\gamma_j\cap\beta_i=\emptyset$ for $i\neq j$.
	After attaching  $D^{\beta}_i$ for $i\neq j$ to $(\Sigma-D)-\bigcup_{i\neq j} h^{\alpha}_i$,  we obtain a once punctured torus $T$.
	$\gamma_i$ is embedded in $T$. Hence $\gamma_i$ is a torus knot.
	We can assume that $\beta_j$ is longitude and $\alpha_j$ is a meridian of $T$.
	Since $\gamma_i$ intersects $\beta_j$ exactly once, $\gamma_i$ is a torus knot $T(1, q)$.
	This implies that $\gamma_i$ is an unknot.
\end{proof}

The following lemma shows that we can calculate the framing coefficient of each component of Kirby diagram obtained from a trisection diagram by using algebraic intersection number.
Particularly, we can characterize the component whose framing coefficient is $0$ by the number of intersections with $\alpha$ and $\beta$.
\begin{lem}\label{framing}
Suppose that a trisection diagram $(\Sigma; \alpha, \beta, \gamma)$ satisfies that $(\Sigma; \alpha, \beta)$ is a standard Heegaard diagram of $S^3$.
Let $L$ be a Kirby diagram obtained from a trisection diagram $(\Sigma; \alpha, \beta, \gamma)$ by the construction above.
Also let $\gamma_i$ be a simple closed curve which is represented by $(a_1, b_2, . . . , a_g, b_g)$ in $H_1(\Sigma; \mathbb{Z})$ respect to the standard basis $\{\alpha_1, \beta_1, . . . . , \alpha_g, \beta_g\}$ of $H_1(\Sigma; \mathbb{Z})$.
Then the framing coefficient of $\gamma_i$ in $L$ is 
\[
		- \sum_{j}(i(\gamma_i, \alpha_j)i(\gamma_i, \beta_j))=- \sum_{j} a_jb_j.
\]
\end{lem}
\begin{proof}
	Let $C$ be a simple closed curve in $\Sigma^\circ=\Sigma-D^2$ which is parallel to $\gamma_i$.
	We orient $C$ by the same orientation of $\gamma_i$.
	The framing coefficient of $\gamma_i$ is a linking number of a link $\gamma_i\cup C$.
	Let $h^{\alpha}_j$ and $h^{\beta}_j$ be  1-handles of $\Sigma^\circ$ which contains $\alpha_j$ and $\beta_j$ respectively for $j=1, . . . , g$.
	If $\gamma_i$ intersects $\alpha_j$ (resp. $\beta_j$), there are sub arcs in $h^{\beta}_j$ (resp. $h^{\alpha}_j$).
	Hence $C$ is also have subarcs in $h^{\beta}_j$ and $h^{\alpha}_j$. 
	The number of subarcs of $\gamma_i$ in $h^{\beta}_j$ and $h^{\alpha}_j$ equals to that of $C$.
	Let $l$ be one of such subarcs of $C$. 
	
	Suppose that $l$ intersects $\beta_j$ with algebraic intersection number $+1$.
	Then we can assume that the sum of the signs of crossings made by $l$ and $\gamma_i$ is $-i(\gamma_i, \alpha_j)$ (See Figure \ref{sign}).
	On the other hand,  if  $l$ intersects $\beta_j$ with algebraic intersection number $-1$, the sum of the sign of crossings made by $l$ and $\gamma_i$ is $i(\gamma_i, \alpha_j)$.
	Hence we consider the sum of the signs of crossings made by sub arcs of $C$ in $h^{\alpha}_j$ and $\gamma_i$, this  equals to $-i(\gamma_i, \alpha_j)i(\gamma_i, \beta_j)$.
	Same as above we consider about the crossing made by sub arcs of $C$ in $h^{\beta}_j$ and $\gamma_i$, this  also equals to $-i(\gamma_i, \alpha_j)i(\gamma_i, \beta_j)$.
	Hence after  summing  all signs of crossings, we obtain 
	\[
		- 2\sum_{j=1}^g i(\gamma_i, \alpha_j)i(\gamma_i, \beta_j).
	\]
	Hence the linking number of $C\cup \gamma_i$ equals to 
	\[
		-\sum_{j=1}^g i(\gamma_i, \alpha_j)i(\gamma_i, \beta_j)=-\sum_{j=0}^g a_jb_j.
	\]
\end{proof}
 \begin{figure}[h]
		\begin{center}
			\includegraphics[scale=0.6]{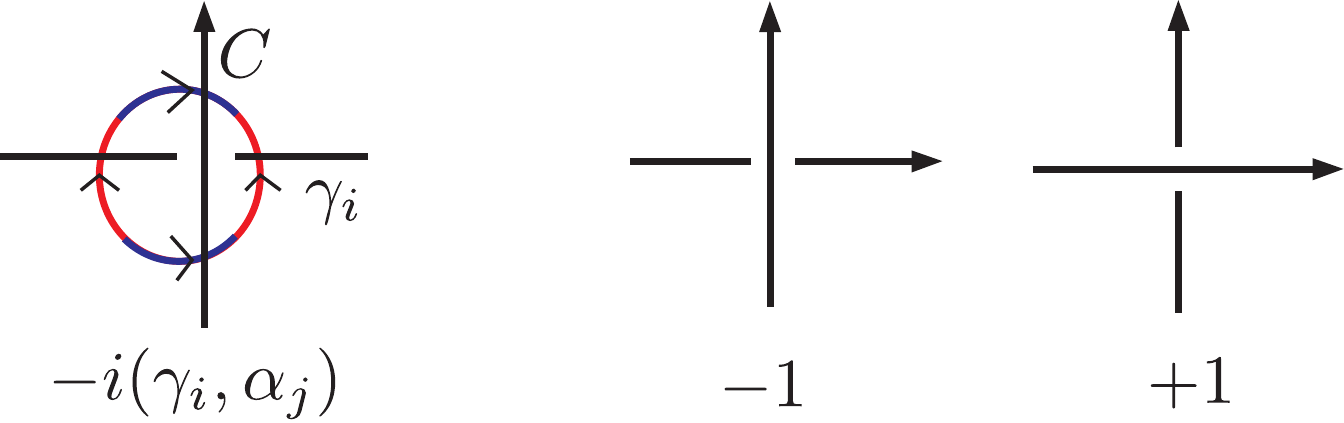}
		\end{center}
			\caption
			{The sign of crossing made by $\gamma$ and $C$ can be calculated as in this figure. The left of this figure shows the case where $C$ intersects $\beta_j$ with algebraic intersection number $+1$.}
			\label{sign}
	\end{figure}
	
	We note that, by considering the basis of $H_1(\Sigma)$ whose orientation is a reverse orientation of standard basis, the framing coefficient of $\gamma_i$ can be calculated by  \[\sum_{j}(i(\gamma_i, \alpha_j)i(\gamma_i, \beta_j))= \sum_{j} a_jb_j.\]
	
	By Lemma \ref{framing}, we can show that if the component $\gamma_i$ is good for $\alpha$ and $\beta$ and satisfies $\gamma_i\mathcal{D}\alpha_j$ and $\gamma_i\mathcal{D}\beta_k$ for $j\neq k$, then the its framing coefficient is $0$.

\section{A proof of the main result}
Let $l$ be a realization loop of a trisection $\mathcal{T}$ with length 2.
Then there are two cases where $\Gamma_{\gamma}\cap l$ has two edges or both $\Gamma_{\gamma}\cap l$ and $\Gamma_{\beta}\cap l$ have exactly one edge.
First of all, we consider the former case.

\subsection{The case where $\Gamma_{\gamma}\cap l$ has two edges }
We consider as $\gamma=(\gamma_1, . . , \gamma_g)$ and $\gamma'=(\gamma_1, . . . , \gamma_{g-1}, \gamma_g')$.
Then there are two cases $\gamma''=(\gamma_1, . . . , \gamma_{g-1}', \gamma_g')$ or $\gamma''=(\gamma_1, . . . , \gamma_{g-1}, \gamma_g'')$. 


Two cut systems on the same surface are {\it slide equivalent} if one can be transformed to the other by a sequence of handle slides and isotopies.
We say two trisection diagram $(\Sigma, \alpha, \beta, \gamma)$ and $(\Sigma, \alpha', \beta', \gamma')$ are {\it slide equivalent} if $\alpha$, $\beta$ and $\gamma$ are slide equivalent to $\alpha'$, $\beta'$ and $\gamma'$ respectively.
We note that two manifolds are diffeomorphic to each other if their trisection diagrams are slide equivalent.
The following lemma shows that we only have to consider the case where $(\Sigma; \alpha, \beta)$ is the standard Heegaard diagram of $S^3$ if the length of a trisection is 2.

 \begin{lem}\label{simple}
 	Let $\mathcal{T}$ be a trisection and $(\Sigma; \alpha, \beta, \gamma)$ a trisection diagram of $\mathcal{T}$.
	Suppose $\alpha$ is good for $\beta$ and $\gamma$.
	If $\alpha_i\mathcal{P}\beta_j$ or  $\alpha_i\mathcal{P}\gamma_k$ for some $i, j, k$,  then $(\Sigma; \alpha, \beta, \gamma)$ is slide equivalent to $(\Sigma'; \alpha', \beta', \gamma')\#(T; \alpha'', \beta'', \gamma'')$  where $(T; \alpha'', \beta'', \gamma'')$ is a genus one trisection diagram of $S^4$ or $S^1\times S^3$.
 \end{lem}
 \begin{proof}
 	We suppose that  $\alpha_i\mathcal{P}\beta_j$ for some $i, j$.
 	Since $\gamma$ is good for $\alpha$, there is a component of $\gamma$ which intersects $\alpha_i$ exactly once or parallel to $\alpha_i$.
	We say $\gamma_1$.
	If $\gamma_1$ is parallel to $\alpha_i$ we can reduce the genus of a trisection diagram and again proceed on the remainder (see case 1 of proof of Theorem 13 of \cite{KT}).
	Hence we can assume that $\gamma_1\mathcal{D} \alpha_i$.
	Then we take a tubular neighborhood $N(\gamma_1\cup \alpha_i)$  of $\gamma_1\cup \alpha_i$  which contains $\beta_j$.
	If there is a component of $\beta$ which intersects $\gamma_1$, we handle sliding such components along $\beta_j$.
	After performing such handle sliding for all such $\beta's$, $N(\gamma_1\cup \alpha_i)$ does not contains components of $\beta$ other than $\beta_j$.
	Now,  $N(\gamma_1\cup \alpha_i)$ is a genus one trisection diagram of $S^4$ with a puncture.
	Hence we complete the proof.
 \end{proof}
 If the length of trisection is 2, we can assume that the realization loop does not contain the edge contained in $\Gamma_{\alpha}$.
Hence we can assume that $\alpha$ is good for both $\gamma$ and $\beta$ if the length of a trisection is 2.

In this subsection, we consider the case where $\Gamma_{\gamma}\cap l$ has two edges from below.
Hence $\Gamma_\beta$ does not contain an edge.
Therefore we can assume that $\gamma_i\mathcal{D}\alpha_i$ for $i=1, . . . , g$ by Lemma \ref{simple} after reordering subindices if $\Gamma_{\gamma}\cap l$ has two edges.

 \begin{lem}\label{region}
 	Let $\mathcal{T}$ be a trisection and $(\Sigma; \alpha, \beta, \gamma)$ a trisection diagram of $\mathcal{T}$.
	Suppose that  $(\Sigma; \alpha, \beta)$ is a standard Heegaard diagram of $S^3$ and $\gamma_i$ and $\gamma_j$ are good for $\alpha$ and $\beta$.
	Then if $\gamma_i\mathcal{D}\beta_j$,  $\gamma_j\mathcal{D}\beta_i$.
 \end{lem}
 \begin{proof}
 	Let $P$ be a planar surface obtained by cutting $\Sigma$ along $\alpha$ and $\beta$.
	We are now consider $\gamma_i\mathcal{D}\alpha_i$ and $\gamma_j\mathcal{D} \alpha_j$ to be satisfied.
 	Since $\gamma_i\mathcal{D}\alpha_i$ and $\gamma_i\mathcal{D}\beta_j$, there is a disk region $R$ whose boundary consists of subarcs of $\alpha_i$, $\alpha_j$, $\beta_i$, $\beta_j$ and $\gamma_i$ (See Figure \ref{region_1}).
	 \begin{figure}[h]
		\centering
			\includegraphics[scale=0.7]{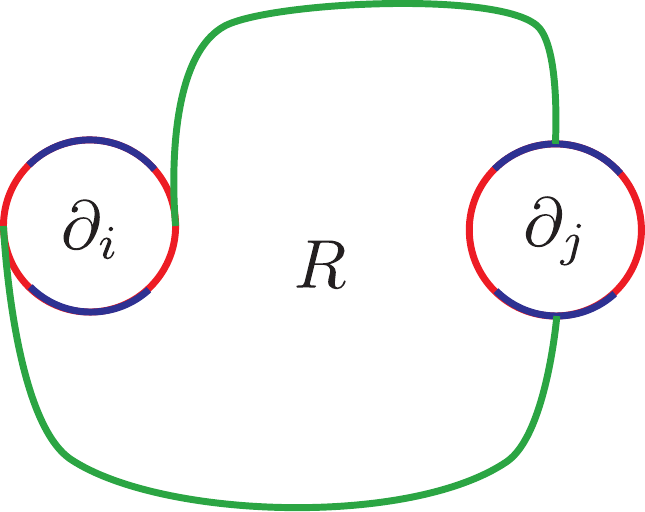}
			\caption
			{If $\gamma_i\mathcal{D}\alpha_i$ and $\gamma_i\mathcal{D}\beta_j$, there is a rectangular $R$ in this figure.}
			\label{region_1}
	\end{figure} 
	$\gamma_j$ have to intersects both $R$ and its exterior since $\gamma_j\mathcal{D}\alpha_j$.
	Since $\gamma_j$ does not intersects $\alpha_k$ and $\beta_i$ for $k\neq j$, $\partial R$ contains $\beta_i$, $\gamma_j$ have to intersects $\beta_i$. 
	Hence $\gamma_j\mathcal{D}\beta_i$ since $\gamma_j$ is good for $\beta$.
 \end{proof}

 \begin{prop}\label{prop_1}
 	Let $\mathcal{T}$ be a trisection of $X$ with $L_{X, \mathcal{T}}=2$, $(\Sigma; \alpha, \beta, \gamma)$ a trisection diagram of $\mathcal{T}$, $l$ a realization loop of $\mathcal{T}$.
	Suppose that $\Gamma_{\gamma}\cap l$ has two edges.
	Let $\gamma=(\gamma_1, \gamma_2, . . . . ,\gamma_{g-1},  \gamma_g)$, $\gamma'=(\gamma_1, \gamma_2, . . . . ,\gamma_{g-1}, \gamma_g')$ and $\gamma''=(\gamma_1, . . . . , \gamma_{g-2},  \gamma_{g-1}', \gamma_g')$ are vertices in $\Gamma_\gamma\cap l$ such that $\gamma '$ is length one from $\gamma$, $\gamma''$ is length one from $\gamma'$ and length two from $\gamma$. 
	Then $X$ is diffeomorphic to the connect-sum of copies of $\cong S^1\times S^3, S^2\times S^2,\pm\mathbb{C}P^2$ and $S^4$.
\end{prop}
\begin{proof}
	We can assume that  $(\Sigma; \alpha, \beta)$ is a standard Heegaard diagram of $S^3$ by Lemma \ref{simple}.
	Hence we can suppose that $\alpha_i\mathcal{D}\beta_i$ for $i=1, . . . , g$.
	By Lemma \ref{simple}, we can assume that $\gamma_i\mathcal{D}\alpha_i$ for $i=1, . . . , g$.
	
	Suppose that one of $\gamma_i$, $\gamma_{g-1}'$ and $\gamma_g'$ is parallel to $\beta_j$.
	We say $\gamma_1$.
	Then we can take a regular neighborhood $T$ of $\alpha_j\cup \beta_j$ which contains $\gamma_1$.
	If $\gamma_i$, $\gamma_{g-1}'$ and $\gamma_g'$  other than $\gamma_1$ is contained in $T$, we perform handlesliding such components along $\gamma_1$.
	After that $T$ does not contain  $\gamma_i$, $\gamma_{g-1}'$ and $\gamma_g'$  other than $\gamma_1$.
	Hence there is an obvious $S^4$ summands in a trisection diagram, which we split off to reduce the genus and again proceed on the remainder.
	
	Let $L$ be a Kirby diagram obtained from $\gamma'$, $L'$ a Kirby diagram obtained from $\gamma''$.
	By Lemma \ref{unknot}, each of $\gamma_i$, $\gamma_{g-1}'$ and $\gamma_g'$ is an unknot component of $L$ and $L'$ for $i=1, . . . , g$ since $\gamma''$ is good for $\beta$.
	The framing coefficient of $\gamma_i$ is $0$ or $\pm 1$ by Lemma \ref{framing} for $i=1, . . . . , g-2$.
	From below, we divide the three cases where $\gamma_g'\mathcal{D}\beta_g$, $\gamma_g'\mathcal{D}\beta_{g-1}$ and $\gamma_g'\mathcal{D}\beta_i$ for $i=1, . . . . , g-2$ and consider each cases. 
	
	First, we consider the case where $\gamma_g'\mathcal{D}\beta_g$.
	Since $\gamma_g\mathcal{D}\alpha_g$, the curves of $\gamma$ other than $\gamma_g$ does not intersects $\alpha_g$.
	Therefore $\gamma_g'$ is always lower than any components of $L$ since the curves of $\gamma'$ do not intersects $\alpha_g$.
	Then  $\gamma_g'$ is separable from any other components in $L$.
	 If $\beta_{g-1}\mathcal{D}\gamma_{g-1}'$,  $\gamma_{g-1}$ is also separable from any other components in $L$ since $\gamma_{g-1}$ is always upper than any other $\gamma_i$ and $\gamma_g'$ for $i=1, . . . , g-2$ in a diagram of $L$.
	 In this case, $\gamma_{g-1}\cup \gamma_g'$ is separable from any components of $L$ and it represents a connect-sum of $\pm\mathbb{C}P^2$s since $X$ is a closed 4-manifold.
	 Suppose $\beta_{g-1}\mathcal{D}\gamma_i$ for one of $i=1, . . . , g-2$.
	 Then  $\gamma_j\mathcal{D}\beta_i$ or $\gamma_{g-1}'\mathcal{D}\beta_i$ for $j\neq g, g-1, i$.
	 Suppose that  $\gamma_j\mathcal{D}\beta_i$.
	 By Lemma \ref{region}, $\gamma_i\mathcal{D}\beta_j$ since $\gamma_i$ and $\gamma_j$ is good for $\alpha$ and $\beta$. This is a contradiction.
	Hence $\beta_{i}\mathcal{D}\gamma_{g-1}'$. 
	Then $\gamma_{g-1}'\cup \gamma_i$ is the Hopf link in $L'$ since there is exactly one crossing such that $\gamma_i$ is lower than $\gamma_{g-1}'$.
	$\gamma_{g-1}'$ is always lower than $\gamma_j$ for $j\neq i, g, g-1$ and $\gamma_i$ does not make a crossing with $\gamma_j$  for $j\neq i, g, g-1$.
	Hence $\gamma_{g-1}'\cup \gamma_i$ is separable from the components of $L'$ except $\gamma_g'$.
	Now, $\gamma_{g}'$ does not intersects $\beta_i$ for $j=1, . . . , g-1$ since $\gamma_g'\mathcal{D}\beta_g$.
	Then $\gamma_g'\cup \gamma_{i}$ is the unlink since $\gamma_g$ is always lower than $\gamma_{i}$ in a diagram of $L'$.
	After handle sliding $\gamma_g'$ along $\gamma_i$, we can assume that each of $\gamma_{g-1}'\cup \gamma_i$ and $\gamma_g'$ is separable from each other since the framing coefficient of $\gamma_i$ is $0$.
	Hence $(\gamma_{g-1}'\cup \gamma_i)\cup \gamma_g'$ represents a connect-sum of $S^2\times S^2$ and $\pm\mathbb{C}P^2$ since $X$ is a closed 4-manifold.
	We note that remnants of $\gamma''$ is good for $\alpha$ and $\beta$.
	Hence remnants of $L'$ represents the 4-manifolds we want to obtain, we can complete the proof in this case.
	
	Next, we consider the case where $\gamma_g'\mathcal{D}\beta_{g-1}$.
	Then $\gamma_{g-1}\cup \gamma_g'$ is the Hopf link in $L$ since there is exactly one crossing of $\gamma_{g-1}\cup \gamma_g'$ such that $\gamma_{g-1}$ is lower.
	$\gamma_g'$ is also separable from $\gamma_i$ for $i=1, . . . , g-2$ in $L$, since $\gamma_g'$ is always lower than $\gamma_i$ for $i=1, . . . , g-2$ in diagram of $L$.
	Also, $\gamma_{g-1}$ is separable from $\gamma_i$ for $i=1, . . . , g-2$ in $L$, since $\gamma_{g-1}$ is always upper than $\gamma_i$ for $i=1, . . . , g-2$ in diagram of $L$.
	Hence $\gamma_{g-1}\cup \gamma_g'$ is separable from any other components in $L$.
	Since $\gamma_{g-1}\cup \gamma_g'$ is the Hopf link and $X$ is a closed 4-manifold, $\gamma_{g-1}\cup \gamma_g'$ represents a connect-sum $S^2\times S^2$ or $\pm\mathbb{C}P^2\#\pm\mathbb{C}P^2$.
	Since remnants of $\gamma'$ are good for $\alpha$ and $\beta$, $L$ represents the 4-manifolds we want to obtain.
	Then we can complete the proof in this case.
	
	Next, we consider the case where $\gamma_g'\mathcal{D}\beta_{i}$ for $i\in \{1, . . . , g-2\}$.
	If $\gamma_i\mathcal{D}\beta_j$ for $j=1, . . . , g-2$, $\gamma_j\mathcal{D}\beta_i$ by Lemma \ref{region}. This is a contradiction.
	Hence either $\gamma_i\mathcal{D}\beta_g$ or  $\gamma_i\mathcal{D}\beta_{g-1}$.
	This implies that $\gamma_g'\cup \gamma_i$ is separable from $\gamma_j$ for $j=1, . . . , g-2$ in $L$ and $L'$.
	Suppose $\gamma_{g-1}'\mathcal{D}\beta_g$. Then $\gamma_i\mathcal{D}\beta_{g-1}$.
	$\gamma_g'\cup \gamma_i$ is the Hopf link in $L$ since there is exactly one crossing of $\gamma_{g}'\cup \gamma_i$ such that $\gamma_{i}$ is lower.
	Since $\gamma_i\mathcal{D}\beta_{g-1}$, $\gamma_i$ is separable from $\gamma_j$ for $j=1, . . . , g-2$ in $L'$.
	$\gamma_g'$ is also separable from $\gamma_j$ for $j=1, . . . , g-2$ in $L'$ except $\gamma_i$.
	$\gamma_{g-1}'$ is separable from $\gamma_j$ for $j=1, . . . , g-2$ in $L'$ since $\gamma_{g-1}'$ is always lower than $\gamma_j$ for $j=1, . . . , g-2$ in diagram of $L'$.
	Particularly, $\gamma_{g-1}'$ is separable from $\gamma_i$ in a diagram of $L'$.
	Hence after handle sliding $\gamma_{g-1}'$ along $\gamma_i$, $\gamma_{g-1}'$ is separable from $\gamma_g'\cup \gamma_i$ since the framing coefficient of $\gamma_i$ is $0$.
	This implies that $(\gamma_g'\cup \gamma_i)\cup \gamma_{g-1}'$ represents a connect-sum of $S^2\times S^2$ and $\pm\mathbb{C}P^2$ since $X$ is a closed 4-manifold if  $\gamma_{g-1}'\mathcal{D}\beta_g$.
	Next, we suppose that $\gamma_{g-1}'\mathcal{D}\beta_{g-1}$.
	Then $\gamma_i\mathcal{D}\beta_g$.
	Then $\gamma_{g-1}$ is separable from $ \gamma_j$ for $j=1, . . . , g-2$ in $L$ since $\gamma_{g-1}'\mathcal{D}\beta_{g-1}$. 
	Also $\gamma_{g-1}$ is separable from $\gamma_i\cup\gamma_g'$ since $\gamma_i\mathcal{D}\beta_g$.
	This implies that $(\gamma_i\cup\gamma_g')\cup \gamma_{g-1}$ in $L$ represents a connect-sum of $S^2\times S^2$ and $\pm\mathbb{C}P^2$ since $X$ is a closed 4-manifold if $\gamma_{g-1}'\mathcal{D}\beta_{g-1}$.
	Finally, we suppose that  $\gamma_{g-1}'\mathcal{D}\beta_j$ for  $j\neq g-1,  g, i$.
	If $\gamma_j\mathcal{D}\beta_k$ for $k=1, . . . , g-2$ except $i$, then $\gamma_k \mathcal{D} \beta_{j}$ by Lemma \ref{region}.
	This is a contradiction.
	Hence $\gamma_j\mathcal{D}\beta_g$ or $\gamma_j\mathcal{D}\beta_{g-1}$.
	See Figure \ref{3prop_1(1)} as the figure of the case where $\gamma_j\mathcal{D}\beta_{g-1}$.
	\begin{figure}[h]
		\centering
			\includegraphics[scale=0.7]{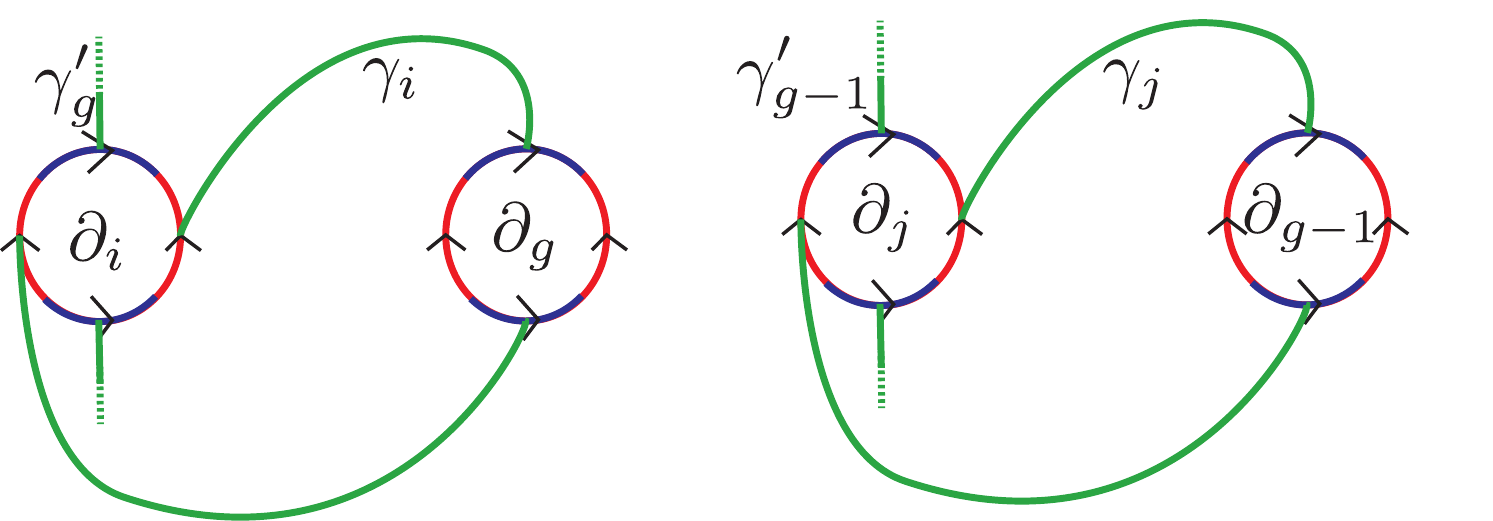}
			\caption
			{If $\gamma_j\mathcal{D}\beta_{g-1}$, $\gamma_i\mathcal{D}\beta_g$ since $\gamma_i$ and $\gamma_j$ is good for $\beta$. This figure is the case where $\gamma_j\mathcal{D}\beta_{g-1}$.}
			\label{3prop_1(1)}
	\end{figure}
	Since there is exactly one crossing made by $\gamma_j$ and $\gamma_{g-1}'$ whose upper arc is $\gamma_{g-1}'$, $\gamma_j\cup \gamma_{g-1}'$ is the Hopf link.
	Same as above $\gamma_i\cup \gamma_g'$ is the Hopf link.
	Suppose $i=1$, $j=2$.
	Since $\gamma_1$ and $\gamma_2$ does not intersects $\beta_1$ and $\beta_2$ respectively, the framing coefficient of $\gamma_1$ is $0$.
	$\gamma_g'$ and  $\gamma_{g-1}'$ are separable from $\gamma_i$ other than $i=1, 2$.
	$\gamma_1\cup\gamma_2$ is an unlink since $\gamma_1$ and $\gamma_2$ does not make a crossing in the diagram of $L'$.
	Hence $(\gamma_j\cup \gamma_{g-1}')\cup (\gamma_i\cup \gamma_g')$ is separable from any other components in $L'$.
	In $(\gamma_j\cup \gamma_{g-1}')\cup (\gamma_i\cup \gamma_g')$, after handlesliding $\gamma_g'$ along $\gamma_1$, $\gamma_i\cup \gamma_g'$ is separable from $\gamma_j\cup \gamma_{g-1}'$ since the framing of $\gamma_1$ is $0$.
	Hence  $(\gamma_j\cup \gamma_{g-1}')\cup (\gamma_i\cup \gamma_g')$ represents a connect-sum of copies of $S^2\times S^2$ and $\pm\mathbb{C}P^2$ since $X$ is a closed 4-manifold.
	Since remnants of $\gamma''$ are good for $\alpha$ and $\beta$, $L'$ represents the 4-manifolds we want to obtain. Then we can complete the proof.

	\end{proof}

	 \begin{prop}\label{prop_2}
 	Let $\mathcal{T}$ be a genus $g$ trisection of $X$ with length 2 and $l$ a realization loop of $\mathcal{T}$.
	Suppose that $\Gamma_\gamma\cap l$ consists of two edges and $\gamma$, $\gamma'$ and $\gamma''$ vertices in $\Gamma_\gamma\cap l$.
	If $\gamma=(\gamma_1, . . . , \gamma_g)$, $\gamma'=(\gamma_1, . . . , \gamma_g')$ and $\gamma''=(\gamma_1, . . . , \gamma_g'')$, then  $X$ is diffeomorphic to connect-sum of copies of $\cong S^1\times S^3, S^2\times S^2,\pm\mathbb{C}P^2$ and $S^4$.
 \end{prop}
	\begin{proof}
		We can assume that  $(\Sigma; \alpha, \beta)$ is a standard Heegaard diagram of $S^3$ by Lemma \ref{simple}.
		Hence we can suppose that $\alpha_i\mathcal{D}\beta_i$ for $i=1, . . . , g$.
		By Lemma \ref{simple}, we can assume that $\gamma_i\mathcal{D}\alpha_i$ for $i=1, . . . , g$.
		Let $L$ be a Kirby diagram obtained from $\gamma''$.
		By Lemma \ref{unknot}, each of $\gamma_i$ for $i=1, . . . . , g-1$ and $\gamma_g''$ is an unknot in $L$.
		
		Suppose that $\gamma_g''\mathcal{D}\beta_g$.
		Then $\gamma_g''$ is always lower than $\gamma_i$ for $i=1, . . . , g-1$ since $\gamma_g''$ does not intersects $\beta_i$ other than $\beta_g$.
		Hence $\gamma_g''$ is separable from $\gamma_i$ for $i=1, . . . , g-1$ in $L$.
		In this case, $\gamma_g''$ represents a connected summand of $\pm \mathbb{C}P^2$ of $X$.
		Since remnants of $\gamma''$ represent the 4-manifolds we want to obtain, we can complete the proof.
		
		Next, we suppose that $\gamma_g''\mathcal{D}\beta_i$ for $i=1, . . . , g-1$.
		If $\gamma_i\mathcal{D}\beta_j$ for $j\neq g, i$, then $\gamma_j\mathcal{D}\beta_i$ by Lemma \ref{region}.
		This is a contradiction.
		Hence $\gamma_i\mathcal{D}\beta_g$.
		Since $\gamma_i$ does not have the crossing with $\gamma_j$ for $j\neq g, i$, $\gamma_i$ is separable from $\gamma_j$ in $L$.
		Also, since $\gamma_g'$ is alway lower than $\gamma_j$ for $j\neq g, i$, $\gamma_g'$ is separable from $\gamma_j$ in $L$.
		Hence $\gamma_g'\cup \gamma_i$ is separable from any other components of $L$.
		Since there is exactly one crossing made by $\gamma_g'$ and $\gamma_i$ such that $\gamma_g'$ is upper, $\gamma_g'\cup \gamma_i$ is the Hopf link in $L$.
		Therefore $\gamma_g'\cup \gamma_i$ represent a connect-sum of copies of $S^2\times S^2$ or $\pm\mathbb{C}P^2$ since $X$ is a closed 4-manifold.
		Since remnants of $\gamma''$ are good for $\alpha$ and $\beta$, $L$ represents the 4-manifolds we want to obtain. Then we can complete the proof.
	\end{proof}
\subsection{The case where both $\Gamma_{\gamma}\cap l$ and $\Gamma_{\beta}\cap l$ has exactly one edge.}
Let $\mathcal{T}$ be a trisection with length 2 and $l$ a realization loop of $\mathcal{T}$.
Also, let $(\Sigma, \alpha, \beta, \gamma)$ be a trisection diagram of $\mathcal{T}$.
Now we can assume that $\alpha$ is good for $\gamma$ and $\beta$.
 
In this subsection, we assume that lengths of trisections are two, and both  $\Gamma_{\gamma}\cap l$ and  $\Gamma_{\beta}\cap l$ have exactly one edge.
Then there are vertices $\gamma'=(\gamma_1, \gamma_2, . . . . ,  \gamma_g')$ and $\beta'=(\beta_1,\beta_2, . . . . , \beta_g')$ which are length one from $\gamma$ and $\beta$ respectively.  
\begin{lem}\label{4.3}
	Let $X$ be a 4-manifolds with trisection $\mathcal{T}$.
	If one of the curves $\gamma_i$ and $\gamma_g'$ in $\gamma'$ is parallel to one of the curves of $\beta'$, then  $X$ is diffeomorphic to connect-sum of copies of $\cong S^1\times S^3, S^2\times S^2,\pm\mathbb{C}P^2$ and $S^4$.
\end{lem}
\begin{proof}
	By Lemma \ref{simple}, we can assume that $(\Sigma, \alpha, \beta)$ is the standard Heegaard diagram of $S^3$.
	Hence we can suppose that $\alpha_i\mathcal{D}\beta_i$ for $i=1, . . . , g$.
	Suppose that one of $\gamma_i$ and $\gamma_g'$ is parallel to $\beta_j$ for $i=1, . . . , g-1$ and $j=1, . . . , g-1$.
	We say $\gamma_1$.
	Then we can take a tubular neighborhood $T$ of $\alpha_j\cup \beta_j$ which contains $\gamma_1$.
	$\gamma_g$ is not contained in $T$ since it is good for $\alpha$ and does not intersect $\gamma_1$.
	If $\gamma_g'$  is contained in $T$, we perform handlesliding such components along $\gamma_1$.
	After that $T$ does not contain  $\gamma_g'$.
	Hence there is an obvious $S^4$ summands in a trisection diagram, which we can split off to reduce the genus and again proceed on the remainder.
	
	Suppose that one of the curves of $\gamma'$ other than $\gamma_g'$ is parallel to $\beta_g'$.
	We say $\gamma_1$.
	Let $L$ be a Kirby diagram obtained from $(\Sigma, \alpha, \beta, \gamma')$.
	Since $\beta_g'$ does not intersect $\beta_1, . . . , \beta_g$, $\gamma_1$ does not intersect  $\beta_1, . . . , \beta_g$.
	This implies that $\gamma_1$ is an unknot component of $L$. 
	Suppose that $\gamma_1\mathcal{D}\alpha_i$ for one of $i=1, . . . , g$.
	We consider the case where $i=g$.
	We can assume  that $\gamma_g'\mathcal{D} \beta_j$ for one of $j=1, . . . , g-1$ since $\gamma_g'$ is good for $\beta'$.
	We note that $j\neq i (=g)$.
	There exists the curve $\gamma_k$ which intersects each of $\alpha_j$ and $\beta_l$ exactly once for one of  $k \neq 1, g$ and $ l\neq g$ since $\gamma$ is good for $\alpha$.
	Then there is the curve $\gamma_m$ of $\gamma'$ for one of $m\neq 1, k, g$ which satisfies $\gamma_m\mathcal{D}\alpha_l$ and it intersects $\beta_i$ or $\beta_j$.
	This contradicts that $\gamma'$ is good for $\beta'$. See Figure \ref{Lem4.3_s1}.
	\begin{figure}[h]
		\centering
			\includegraphics[scale=0.7]{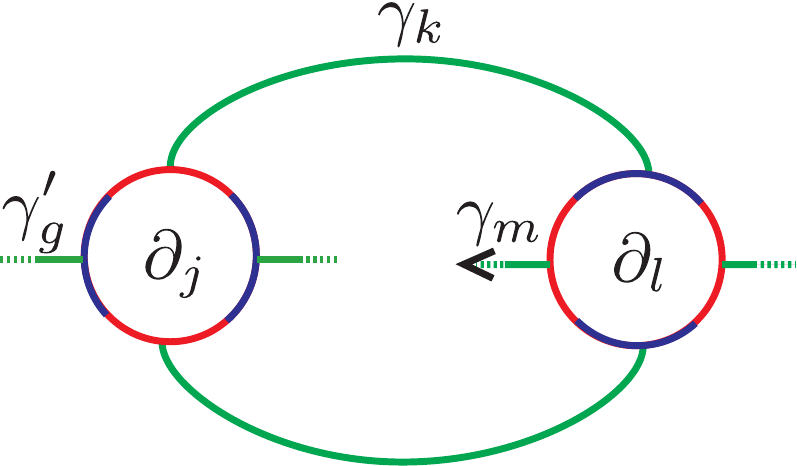}
			\caption
			{If $i=g$, one of the curves of $\alpha$ intersects curves of $\gamma$ twice other than $\gamma_g$.}
			\label{Lem4.3_s1}
	\end{figure}
	Next, we consider the case where $i=1, . . . , g-1$.
	If $\gamma_j\mathcal{D}\beta_i$ for one of $j\neq 1$, $\gamma_1$ have to intersects $\beta_k$ for one of $k=1, . . . , g$ since $\gamma_1$ is good for $\alpha$.
	This is a contradiction.
	This implies that $\gamma_g'\mathcal{D}\beta_i$. 
	Now, $\gamma_g'$ is always lower than any other components of $\gamma'$ other than $\gamma_1$.
	Also, $\gamma_1$ is always lower than any other components of $\gamma'$.
	Hence $\gamma_g'\cup \gamma_1$ is two components unlink in $L$ separable from any other components.
	Since remnants of $\gamma'$ are good for $\alpha$ and $\beta$, $L$ represents the 4-manifolds we want to obtain.

	Next, we consider the case where $\gamma_g'$ is parallel to $\beta_g'$.
	Since $\beta_g'$ does not intersects $\beta_1, . . . , \beta_g$, $\gamma_g'$ does not intersect $\beta_1, . . . , \beta_g$.
	This implies that $\gamma_g'$ is an unknot component of $L$ and separable from any other components of $L$.
	Suppose $\gamma_g\mathcal{D}\alpha_i$ for one of  $i=1, . . . , g$.
	Then, since $\gamma_1, . . . , \gamma_{g-1}$ are good for $\alpha$ and $\beta'$, $\gamma_i$ does not intersects $\alpha_i$ for $i=1, . . . , g-1$.
	Let $\Sigma'=\Sigma-N(\alpha_i)\cup D^2\cup D^2$ be a cupping off $\Sigma-N(\alpha_i)$ by two disks and $\alpha'=\alpha-\alpha_i$, $\beta'=\beta-\beta_i$ and $\gamma'=(\gamma_1, . . . , \gamma_{g-1})$.
	Then $(\Sigma', \alpha', \beta', \gamma')$ is a trisection diagram of length 0 trisection.
	This implies that $L$ represents the 4-manifolds we want to obtain.
\end{proof}

Only remains are the cases where  $\gamma_i\mathcal{D} \beta_g'$ for one of  $i=1, . . . , g-1$ and $\gamma_g'\mathcal{D} \beta_g'$.
First, we consider the case where $\gamma_i\mathcal{D} \beta_g'$ for one of $i=1, . . . , g-1$. We can assume that the curves of $\gamma'$ other than $\gamma_i$ intersects one of the curves of $\beta'$ exactly once by Lemma \ref{4.3}.

\begin{lem}\label{lem4}
	Let $X$ be a 4-manifolds with trisection $\mathcal{T}$.
	If  $\gamma_i \mathcal{D} \beta_g'$ for one of $i\in \{1, . . . , g-1\}$, then $X$ is diffeomorphic to connect-sum of copies of $\cong S^1\times S^3, S^2\times S^2,\pm\mathbb{C}P^2$ and $S^4$.
\end{lem}
\begin{proof}
	By Lemma \ref{simple}, we can assume that $(\Sigma, \alpha, \beta)$ is the standard Heegaard diagram of $S^3$.
	Let $L$ and $L'$ be a Kirby diagram obtained from $\gamma$ and $\gamma'$ respectively.
	We can assume that   $\gamma_1\mathcal{D} \beta_g'$ without loss of generality.
	Since $\gamma$ is good for $\alpha$, there is a curve $\alpha_i$ such that $\gamma_1\mathcal{D}\alpha_i$.
	If $i=g$, $\gamma_1$ is always lower than any other components of $L$ since $\gamma_1$ does not intersect $\beta_i$ for $i=1, . . . , g-1$. 
	Hence $\gamma_1$ is separable from any other components of $L$.
	Remnants of $\gamma$ is good for $\alpha-\alpha_g$ and $\beta-\beta_g$ and does not intersect $\alpha_g$.
	We note that the arcs that intersect $\beta_g$ do not make a crossing of $L-\gamma_1$.
	Hence remnants of $\gamma$ represent the 4-manifolds we want to obtain. 
	We can complete the proof in this case.
	
	If $i\in \{1, . . . , g-1\}$, there is a curve $\gamma_j$ or $\gamma_g'$ that intersects $\beta_i$ exactly once.
	
	Suppose that $\gamma_g'\mathcal{D}\beta_i$.
	Then $\gamma_g'\cup \gamma_1$ is the Hopf link in $L'$ that is separable from any other components of $L'$.
	Let $\gamma_k$ be a curve that intersects $\alpha_g$ exactly once.
	Then $\gamma_k$ intersects $\beta_m$ for $m\neq i, g$ exactly once.
	If $\gamma_g$ intersects  $\alpha_m$ exactly once, $\gamma_k$ is separable from any other components of $L'$.
	Suppose that $\gamma_m$ intersects $\alpha_m$ for $m\neq 1, k, g$.
	Then $\gamma_m$ intersects $\beta_l$ exactly once for $l\neq i, m, g$.
	If $\alpha_l\mathcal{D} \gamma_l$ for $l\neq g$, this contradicts Lemma \ref{region}.
	Hence $\alpha_l\mathcal{D}\gamma_g$.
	This implies that $\gamma_k\cup \gamma_m$ is the Hopf link in $L'$ that is separable from any other components. 
	Since remnants of $\gamma'$ are good for $\alpha$ and $\beta$, $L'$ represents the 4-manifolds we want to obtain. Then we can complete the proof in this case.
	
	We consider the case where $\gamma_j\mathcal{D}\beta_i$ for one of $j=2, . . . , g-1$.
	Then $\gamma_j\mathcal{D}\alpha_k$ for $k\neq i, j$.
	Suppose that $k=g$.
	Same as above, $\gamma_j\cup \gamma_1$ is a Hopf link in $L$ which is separable from any other components.
       	$\gamma_g'$ intersects one of the curves of $\beta'$ exactly once.
	We say $\beta_m$.
	Then there is a curve of $\gamma'$ which intersects $\alpha_m$ exactly once.
	If such curve intersects $\beta_l$ exactly once, $\gamma_g\mathcal{D}\alpha_l$ by Lemma \ref{region}.
	The union of the such curve and $\gamma_g$ is a Hopf link in $L$ which is separable from any other components.
	Since remnants of $\gamma$ are good for $\alpha$ and $\beta$, $L$ represents the 4-manifolds we want to obtain.
	Next, we suppose that $k\neq g$.
	Then $\gamma_1$ have to intersects $\beta_k$ since $\gamma_1$ is good for $\alpha$.
	This contradicts that $\gamma_1$ is good for $\beta'$. (See Figure \ref{Lem4.4_s1}). 
	\begin{figure}[h]
		\centering
			\includegraphics[scale=0.7]{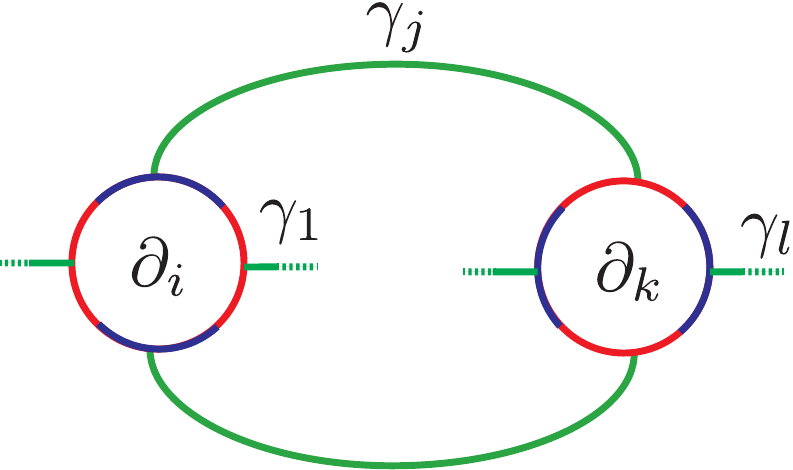}
			\caption
			{If $k\neq g$ and $\gamma_j\mathcal{D}\beta_i$, $\gamma_1$ have to intersects $\beta_k$.}
			\label{Lem4.4_s1}
	\end{figure}


	\end{proof}
	\begin{lem}\label{lem5}
		Let $X$ be a 4-manifolds with trisection $\mathcal{T}$.
		If $\gamma_g'\mathcal{D}\beta_g'$, then $X$ is diffeomorphic to connect-sum of copies of $\cong S^1\times S^3, S^2\times S^2,\pm\mathbb{C}P^2$ and $S^4$.
	\end{lem}
	\begin{proof}
		By Lemma \ref{simple}, we can assume that $(\Sigma, \alpha, \beta)$ is the standard Heegaard diagram of $S^3$.
		Let $L$ and $L'$be Kirby diagrams obtained from $\gamma$ and $\gamma'$ respectively.
		Now $\gamma_i$ is good for $\alpha$ and $\beta'$.
		Hence we can assume that $\gamma_i$ intersects one of the curves of $\beta'$ other than $\beta_g'$, and $\alpha_k$ exactly once and $\beta_g$ several times for $k=1, . . . , g$, $i =1, . . . , g-1$.
		Let $\{ \alpha_1, \beta_1, . . . , \alpha_g, \beta_g\}$ be a basis of $H_1(\Sigma)$.
		Now, $\beta_g'$ does not intersects $\beta_i$ for $i=1, . . . . , g$.
		Then we can suppose that 
		 \begin{alignat}{1}
			\beta_g'&=(0, z^g_2, 0, z^g_4, 0, . . . . 0, z^g_{2g})\nonumber
		\end{alignat}
		in $H_1(\Sigma)$ since $i(\beta_g', \beta_i)=0$ for $i=1, . . . , g$.
		Since $\gamma_g'$ does not intersects $\beta_i$ for $i=1, . . . , g-1$, we can suppose that 
		\begin{alignat}{1}
			\gamma_g'&=(0, y^g_2, 0, y^g_4, 0, . . . . , y^g_{2g-1}, y^g_{2g})\nonumber
		\end{alignat}
		in $H_1(\Sigma)$.
		Since $\beta_g'$ and $\gamma_g'$ intersect at exactly one point, we can assume that 
		\[
			i(\beta_g', \gamma_g')=z^g_{2g}y^g_{2g-1}=1
		\]
		by orientating $\gamma_g$ to be so.
		This implies that $(z^g_{2g}, y^g_{2g-1})=(\pm 1, \pm 1)$.
		Then we shall obtain
		 \begin{alignat}{1}
			\beta_g'&=(0, z^g_2, 0, z^g_4, 0, . . . . , 0, \pm 1)\nonumber \\
			\gamma_g'&=(0, y^g_2, 0, y^g_4, 0, . . . . , \pm 1, y^g_{2g})\nonumber
		\end{alignat}

		\begin{claim}\label{cl1}
			$\gamma_g'$ is an unknot.
		\end{claim}
		
		\begin{proof}[Proof of Claim]
			In the construction of the Kirby diagram in Section 4, the disks attached to the boundary of $\Sigma-D^2$ along $\beta$ were $D_1, ... . D_g$.
			Let $h_i$ be a 2-dimensional 2-handle of $\Sigma-D^2$ which contains $\alpha_i$ for $i=1, . . . . , g-1$.
			$\gamma_g'$ can be embedded in a punctured torus $T=((\Sigma-D^2)-(\cup_i h_i)) \cup D_1\cup D_2\cup\cdots \cup D_{g-1}$ since $\gamma_g'$ does not intersects $\beta_i$ for $i=1, . . . . , g-1$.
			Now we have
			\[
				\gamma_g'=(0, y^g_2, 0, y^g_4, 0, . . . . , \pm 1, y^g_{2g}).
			\]
			We can assume that $\alpha_g$ is a meridian of $T$ and $\beta_g$ a longitude of $T$.
			Since $i(\gamma_g', \alpha_g)=\pm 1$ and $i(\gamma_g', \beta_g)= y^g_{2g}$, $\gamma_g'$ is a torus knot $T(\pm 1, y^g_{2g})$. This is an unknot in a diagram $L'$.
			This implies that $\gamma_g'$ is an unknot component in $L'$.
		\end{proof}
		Suppose  $\gamma_j\mathcal{D}\alpha_g$ and $\gamma_j\mathcal{D}\beta_k$ for $j, k\neq  g$.
		Then, $\gamma_k\mathcal{D}\alpha_k$ and $\gamma_k\mathcal{D}\beta_l$ for $l\neq  j, k, g$ or $\gamma_g\mathcal{D}\alpha_k$.
		If $\gamma_g\mathcal{D}\alpha_k$, $\gamma_g\cup \gamma_j$ is the Hopf link in $L$ that is separable from any other components of $L$.
		Since remnants of $\gamma$ are good for $\alpha$ and $\beta$, $L$ represents the 4-manifolds we want to obtain. Then we can complete the proof in this case.
		
		Suppose  $\gamma_k\mathcal{D}\alpha_k$ and $\gamma_k\mathcal{D}\beta_l$ for $k\neq j, g$ and $l\neq  j, k, g$
		If $\alpha_l\mathcal{D}\gamma_l$ for $l\neq g$, This contradicts Lemma \ref{region}.
		Hence $\alpha_l\mathcal{D} \gamma_g$.
		If $\gamma_k$ intersects $\beta_g$ more than twice, $\gamma_k$ have to intersects $\alpha_g$.
		This contradicts that $\gamma_k$ is good for $\alpha$.
		Since $\gamma_k$ does not intersects $\alpha_g$ and good for $\alpha$ and $\beta'$, the framing coefficient of $\gamma_k$ is 0 by Lemma \ref{framing}.
		Now,  $\gamma_j\cup \gamma_k\cup \gamma_g$ is separable from any other components of $L$ since $\gamma_k$ is upper than any components of $L$ other than $\gamma_j$ and $\gamma_g$.
		Since $\gamma_k$ intersects $\beta_g$ exactly once, $\gamma_j\cup \gamma_k$ is the Hopf link in $L'$.
		Also, since $\gamma_g$ is always upper than $\gamma_j$, $\gamma_g\cup \gamma_j$ is unlink in $L'$.
		Furthermore, since $\gamma_k$ intersects $\beta_l$ exactly once, $\gamma_k\cup \gamma_g$ is the Hopf link in $L'$.
		Hence, after handlesliding $\gamma_j$ along $\gamma_g$, and handlesliding $\gamma_g$ along $\gamma_k$ finitely many times,  $\gamma_g\cup \gamma_j\cup \gamma_k$ is a union of the unknot and the Hopf link they are separable from each other in $L'$.
		Also, they are separable from any other components of $L$.
		This completes the proof in this case.
		
		Suppose that $\gamma_g\mathcal{D}\alpha_g$.
		Then $\gamma_g'$ is separable from any other component since $\gamma_g\mathcal{D}\alpha_g$ in $L'$.
		Therefore $\gamma_g'$ is an unknot that is separable from any other components of  $L$ since $\gamma_g'$ is always lower than any other components of $L$.
		Let $\Sigma'=\Sigma-N(\alpha_g)\cup D^2\cup D^2$.
		$\Sigma'$ is capping off of $\Sigma-N(\alpha_g)$ by two disks.
		Then $\gamma_j$ for $j=1, . . . , g-1$ does not intersects $\alpha_g$ since $\gamma'$ is good for $\alpha$.
		Let $\alpha'=\alpha-\alpha_g$, $\beta''=\beta-\beta_g$ and $\gamma''=\gamma'-\gamma_g'$.
		Then $(\Sigma', \alpha', \beta'', \gamma'')$ is a trisection diagram of $X$ with length 0.
		Hence remnants of $\gamma'$ represent the 4-manifolds we want to obtain.
	\end{proof}	
	Now, By Lemma \ref{lem4} and \ref{lem5} we obtain the following proposition.
	\begin{prop}\label{prop_3}
		Let $\mathcal{T}$ be a trisection of $X$ with $L_{X, \mathcal{T}}=2$ and $l$ a realization loop of $\mathcal{T}$.
		If both $\Gamma_{\gamma}\cap l$ and $\Gamma_{\beta}\cap l$ have exactly one edge, then $X$ is diffeomorphic to connect-sum of copies of $\cong S^1\times S^3, S^2\times S^2,\pm\mathbb{C}P^2$ and $S^4$.
	\end{prop}
	
	\begin{proof}[Proof of Theorem \ref{thm1}]
		 Let $\mathcal{T}$ be a trisection, $(\Sigma, \alpha, \beta, \gamma)$  a trisection diagram of $\mathcal{T}$ and $l$ a realization loop of $\mathcal{T}$.
		 If the length of trisection $\mathcal{T}$  equals 2, it will be sufficient to consider the following cases.
		 \begin{enumerate}
 				\item $\Gamma_\gamma\cap l$ consists of two edge and $\alpha=\alpha_\beta=\alpha_\gamma$ and $\beta=\beta_\alpha=\beta_\gamma$.
				\item Both $\Gamma_\gamma\cap l$ and $\Gamma_\beta\cap l$ are consists of one edge and $\alpha=\alpha_\beta=\alpha_\gamma$.
		 \end{enumerate}
		By Proposition \ref{prop_1} and \ref{prop_2}, we can obtain the conclusion of theorem in the case where $\Gamma_{\gamma}\cap l$ has two edges.
		By Proposition \ref{prop_3},  we can obtain the conclusion of the theorem in the case where both $\Gamma_{\gamma}\cap l$ and $\Gamma_{\beta}\cap l$ has exactly one edge.
	\end{proof}
	
	There is a 4-manifold $X$ with $L_X\leq6$ and $L_X\neq 0$ named $Q$ in \cite{KT}. The remainder of the work is to investigate what kind of 4-manifold trisections can be trisections such that the length is 3, 4, or 5.
	
\section{Acknowledgements}
 		The author thanks his supervisor Koya Shimokawa for meaningful comments and discussions. Also, we thanks Hironobu Naoe for his helpful comments on a Kirby diagram in this paper.
		The author was partially supported by Grant-in-Aid for JSPS Research Fellow from JSPS KAKENHI Grant Number JP20J20545.


\end{document}